\documentclass[a4paper,12pt]{amsart}
\usepackage[utf8]{inputenc}
\usepackage{amsmath}%
\usepackage{amsfonts}%
\usepackage{amssymb}%
\usepackage{graphicx}
\usepackage{color}
\usepackage{enumerate}
\usepackage{stmaryrd}
\theoremstyle{plain}
\newtheorem{theorem}{Theorem}
\newtheorem{lemma}{Lemma}
\newtheorem{corollary}{Corollary}
\newtheorem{proposition}{Proposition}
\theoremstyle{definition}
\newtheorem{definition}{Definition}
\newtheorem{notation}{Notasjon}
\theoremstyle{remark}
\newtheorem{result}{Result}
\newtheorem{remark}{Remark}
\newtheorem{observation}{Observation}
\numberwithin{equation}{section}

\newcommand{\dd}{\mathrm{\,d}}

\DeclareMathOperator{\diag}{diag}
\DeclareMathOperator{\di}{div}

\begin{document}
\title[]{Superposition in the p-Laplace equation.}
\author{Karl K. Brustad}
\date{January 18, 2016}

\begin{abstract}
That a superposition of fundamental solutions to the \(p\)-Laplace Equation is \(p\)-superharmonic -- even in the non-linear cases \(p>2\) -- has been known since M. Crandall and J. Zhang published their paper \textit{Another Way to Say Harmonic} in 2003. We give a simple proof and extend the result by means of an explicit formula for the \(p\)-Laplacian of the superposition.
\end{abstract}
\maketitle

\section{Introduction}

%It was discovered by M. Crandall and J. Zhang in (...) that the Newtonian potential
%\[V(x) = \int_{\mathbb{R}^n}\frac{\rho(y)}{|x-y|^{n-2}}\dd y\]
%in classical potential theory has a counterpart for the \(p\)-Laplace equation
%\begin{equation}
%\Delta_p u := \di \left(|\nabla u|^{p-2}\nabla u\right) = 0.
%\label{int.eq}
%\end{equation}
%Indeed, the function
%\[V(x) = \sgn(n-p)\int_{\mathbb{R}^n}\frac{\rho(y)}{|x-y|^\frac{n-p}{p-1}}\dd y,\qquad 2<p<\infty,\; \rho \geq 0\]
%is a supersolution to the \(p\)-Laplace equation, i.e., \(\Delta_p V \leq 0\). The same holds for the superposition
%\[\sgn(n-p)\sum_i\frac{a_i}{|x-y_i|^\frac{n-p}{p-i}},\qquad 2<p<\infty,\; a_i > 0\]
%of the fundamental solutions \(|x-y_i|^\frac{p-n}{p-1}\).
%See also (..) for a proof and exact definitions. (In (...) the superposition has been further extended to the corresponding non-linear equation in the Heisenberg group.)
%
%The objective of our note is present an explicit expression for the \(p\)-Laplace of
%\[\sum_i\frac{a_i}{|x-y_i|^\frac{n-p}{p-i}}\]
%and thereby simplify the proof of the superposition. It also becomes clear that we can generalize the superposition to include sums like
%\[\sum_i\frac{a_i}{|x-y_i|^\frac{n-p}{p-i}} + K(x)\]
%where \(K\) is an arbitrary concave function.

Our object is a superposition of fundamental solutions for the \(p\)-Laplace Equation
\begin{equation}
\Delta_p u := \di \left(|\nabla u|^{p-2}\nabla u\right) = 0.
\label{int.eq}
\end{equation}
Although the equation is non-linear, the function
\[V(x) = \int_{\mathbb{R}^n}\frac{\rho(y)}{|x-y|^\frac{n-p}{p-1}}\dd y,\qquad \rho\geq 0,\quad 2\leq p < n\]
is a supersolution in \(\mathbb{R}^n\), i.e. \(\Delta_p V\leq 0\) in the sense of distributions. It is a so-called \(p\)-superharmonic function -- see Definition \ref{def.psup} on page \pageref{def.psup} -- according to which it has to obey the comparison principle.
The case \(p=2\) reduces to the Laplace Equation \(\Delta u = 0\) with the Newtonian potential
\[V(x) = \int_{\mathbb{R}^n}\frac{\rho(y)}{|x-y|^{n-2}}\dd y,\]
which is a superharmonic function.

M. Crandall and J. Zhang discovered in \cite{crandall2003} that the sum
\[\sum_{i=1}^N\frac{a_i}{|x-y_i|^\frac{n-p}{p-1}},\qquad a_i>0\]
of fundamental solutions is a \(p\)-superharmonic function. Their proof was written in terms of viscosity supersolutions. A different proof was given in \cite{lindqvist2008}. The purpose of our note is a \emph{simple} proof of the following theorem:

\newpage
\begin{theorem}\label{int.mainthm}
Let \(2\leq p < n\). For an arbitrary concave function \(K\),
\begin{equation}
W(x) := \sum_{i=1}^\infty \frac{a_i}{|x-y_i|^\frac{n-p}{p-1}} + K(x),\qquad y_i\in\mathbb{R}^n,\, a_i \geq 0, 
\label{int.main}
\end{equation}
is \(p\)-superharmonic in \(\mathbb{R}^n\), provided the series converges at some point.
\end{theorem}

Through Riemann sums one can also include potentials like
\[\int_{\mathbb{R}^n}\frac{\rho(y)}{|x-y_i|^\frac{n-p}{p-1}}\dd y + K(x),\qquad \rho\geq 0.\]

Similar results are given for the cases \(p=n\) and \(p>n\) and, so far as we know, the extra concave term \(K(x)\) is a new feature. The key aspect of the proof is the explicit formula \eqref{sup.sign} for the \(p\)-Laplacian of the superposition. Although the formula is easily obtained, it seems to have escaped attention up until now.

Finally, we mention that in \cite{garofalo2010} the superposition of fundamental solutions has been extended to the \(p\)-Laplace Equation in the Heisenberg group. (Here one of the variables is discriminated.) In passing, we show in Section \ref{ep} that similar results are \emph{not} valid for the evolutionary equations
\[\frac{\partial}{\partial t}u = \Delta_p u\qquad\text{and}\qquad \frac{\partial}{\partial t}(|u|^{p-2}u) = \Delta_p u\]
where \(u = u(x,t)\).
We are able to bypass a lenghty calculation in our counter examples.
%\newpage

%The exact result is as follows:
%
%
%\begin{theorem}\label{int.mainthm}
%Let \(2<p<\infty,\;n\geq2\) and let \(w\) be the fundamental solution to the \(p\)-Laplace equation \eqref{int.eq}.  If given sequences \((a_i)\subseteq (0,\infty)\) and \((y_i)\subseteq \mathbb{R}^n\) where \(\sum_i a_i<\infty\) and \((|y_i|)\) is bounded, then the (extended) function \(W\colon \mathbb{R}^n\rightarrow (-\infty,\infty]\) defined by
%\begin{equation}
%W(x) := \sum_{i=1}^\infty a_iw(x-y_i) + K(x),\qquad K \text{ concave,}
%\label{int.main}
%\end{equation}
%is \(p\)-superharmonic in any bounded domain in \(\mathbb{R}^n\) provided the sum converges at least at one point.
%\end{theorem}
%
%
%While the proof -- which is postponed to section \ref{psup} -- is somewhat technical and lengthy, the main idea is simple and consists of Lemma \ref{sup.thm} in section \ref{sup} and Lemma \ref{add.thm} in section \ref{add}. Finally, section \ref{ep} is added to present some negative results regarding the parabolic case.
%
%\begin{remark}
%Through the limit of Riemann sums, it is clear that Theorem \ref{int.mainthm} can be extended to hold for functions like
%\[\int_{\mathbb{R}^n}\rho(y)w(x-y)\dd y + K(x),\qquad 0\leq\rho \in L_0(\mathbb{R}^n).\]
%\end{remark}

\section{The fundamental solution}
Consider a radial function, say
\[f(x) = v(|x|)\]
where we assume that \(v\in C^2(0,\infty)\). By differentiation
\begin{align}\label{rad.formulas}
\nabla f &= \frac{v'}{|x|}x^T, & |\nabla f| &= |v'|,\\ \notag
\mathcal{H}f &= v''\frac{xx^T}{|x|^2} + \frac{v'}{|x|}\left(I - \frac{xx^T}{|x|^2}\right), & \Delta f &= v'' + (n-1)\frac{v'}{|x|},
\end{align}
when \(x\neq0\).

The Rayleigh quotient formed by the Hessian matrix \(\mathcal{H}f = \left[\frac{\partial^2f}{\partial x_i\partial x_j}\right]\) above will play a central role. Notice that for any non-zero \(z\in\mathbb{R}^n\), we have that
\[\frac{z^T}{|z|}\frac{xx^T}{|x|^2}\frac{z}{|z|} = \cos^2\theta\]
where \(\theta\) is the angle between the two vectors \(x\) and \(z\). This yields the expedient formula
\begin{equation}
\frac{z^T(\mathcal{H}f) z}{|z|^2} = v''\cos^2\theta + \frac{v'}{|x|}\sin^2\theta,\qquad x,z\neq 0.
\label{obs.ray}
\end{equation}

Since the gradient of a radial function is parallel to \(x\), the Rayleigh quotient in the identity
\begin{equation}
\di \left(|\nabla f|^{p-2}\nabla f\right) = |\nabla f|^{p-2}\left((p-2)\frac{\nabla f (\mathcal{H}f) \nabla f^T}{|\nabla f|^2} + \Delta f\right)
\label{int.id}
\end{equation}
reduces to \(v''\). The vanishing of the whole expression is then equivalent to
\begin{equation}
(p-1)v'' + (n-1)\frac{v'}{|x|} = 0
\label{fund.lap}
\end{equation}
which, integrated once, implies that a radially decreasing solution \(w\) is on the form
\begin{equation}
w(x) = v(|x|)\qquad \text{where}\qquad v'(|x|) = -c|x|^\frac{1-n}{p-1}.
\label{fund.sol}
\end{equation}
The constant \(c = c_{n,p}>0\) can now be chosen so that
\[\Delta_p w + \delta = 0\]
in the sense of distributions. Thus
\begin{equation}
w(x) =
\begin{cases}
-c_{n,p}\frac{p-1}{p-n}|x|^\frac{p-n}{p-1}, &\text{when } p\neq n,\\
-c_{n,n}\ln|x|, &\text{when } p = n
\end{cases}
\label{fund.sol2}
\end{equation}
is the \textbf{fundamental solution} to the \(p\)-Laplace Equation \eqref{int.eq}.
%Now
%\[\int |\nabla w|^{p-2}\nabla w\nabla \phi^T\dd x =  \phi(0)\]
%for all test functions \( \phi\in C^\infty_0(\mathbb{R}^n)\).

%Being autonomous, the \(p\)-Laplace equation is clearly invariant under translations of the domain. It is also important to remember that the formulas in \eqref{rad.formulas} and \eqref{fund.lap} and the result in \eqref{rad.ray} remains true even when \(x\) is replaced with \(x-y\).

\section{Superposition of fundamental solutions}\label{sup}
%If the goal is to show that a combination \(V\) of fundamental solutions is \(p\)-superharmonic in \(\mathbb{R}^n\), the natural first question is whether \(\Delta_p V \leq 0\) on the combinations domain in the classical sense. The following lemma -- though it is slightly more general as it handles all dimensions \(n\geq 1\) and all \(p\neq1\) \footnote{When \(p=1\) there are no non-constant radial solutions of \eqref{int.eq}. Instead, the equation reduces to the... } in one-go -- answers this positively and the proof is simply to present an identity for which the sign is obvious.

We now form a superposition of translates of the fundamental solution and compute its \(p\)-Laplacian. To avoid convergence issues all sums are, for the moment, assumed finite.

\begin{lemma}\label{sup.thm}
Let \(w\) be the fundamental solution to the \(p\)-Laplace equation. Define the function \(V\) as
\begin{equation}
V(x) := \sum_{i=1}^N a_i w(x-y_i),\qquad a_i>0,\;y_i\in\mathbb{R}^n.
\label{sup.lincomb}
\end{equation}
Then, in any dimension and for any \(p\neq1\)
\footnote{When \(p=1\) there are no non-constant radial solutions of \eqref{int.eq}. Instead we get the zero mean curvature equation in which a solution's level sets are minimal surfaces.},
\(\Delta_pV\) is of the same sign wherever it is defined in \(\mathbb{R}^n\). Furthermore, the dependence of the sign on \(p\) and \(n\) is as indicated in figure \ref{supfig}.
\end{lemma}

\begin{figure}[ht]
\includegraphics{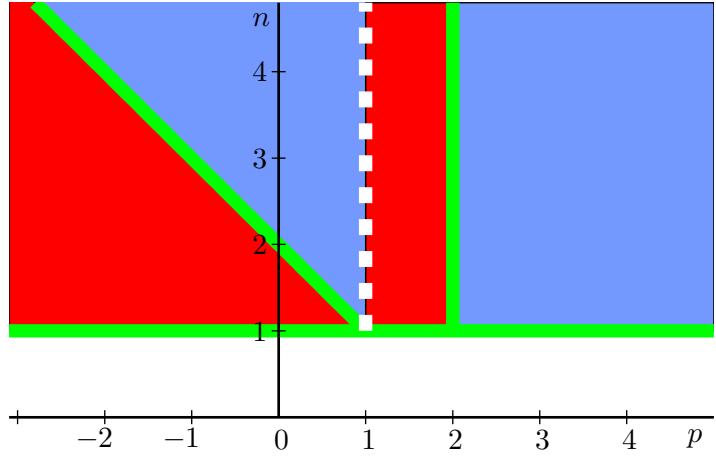}%
\caption{\textcolor{blue}{\(\Delta_p V\leq0\)}, \textcolor{green}{\(\Delta_p V=0\)}, \textcolor{red}{\(\Delta_p V\geq0\)}}%
\label{supfig}%
\end{figure}

\begin{proof}
We simplify the notation by letting \(w_i\) and \(v_i\) denote that the functions \(w\) and \(v\) are to be evaluated at \(x-y_i\) and \(|x-y_i|\), respectively.

First, the linearity of the Hessian and the Laplacian enable us to write
\begin{align*}
\Delta_p V &= |\nabla V|^{p-2}\left((p-2)\frac{\nabla V(\mathcal{H} V)\nabla^T V}{|\nabla V|^2} + \Delta V\right)\\
             &= |\nabla V|^{p-2}\sum_{i=1}^N a_i\left((p-2)\frac{\nabla V(\mathcal{H}w_i)\nabla^T V}{|\nabla V|^2} + \Delta w_i\right).\\
\intertext{Secondly, by \eqref{rad.formulas} and \eqref{obs.ray} this is}
						 &= |\nabla V|^{p-2}\sum_{i=1}^N a_i \left((p-2)\Big(v_i''\cos^2\theta_i + \frac{v_i'}{|x-y_i|}\sin^2\theta_i\Big)\right.\\
						 & \qquad\qquad\qquad\qquad{}+ \left.v_i'' + (n-1)\frac{v_i'}{|x-y_i|}\right)\\
						 &= |\nabla V|^{p-2}\sum_{i=1}^N a_i\left((p-2)\Big(\frac{v_i'}{|x-y_i|} - v_i''\Big)\sin^2\theta_i\right.\\
						 & \qquad\qquad\qquad\qquad{}+ \left.(p-1)v_i'' + (n-1)\frac{v_i'}{|x-y_i|}\right)\\
\intertext{where \(\theta_i\) is the angle between \(x-y_i\) and \(\nabla V(x)\). And finally, as \(w\) is a fundamental solution, the last two terms disappear by \eqref{fund.lap}. We get }
\Delta_p V	 &= (p-2)|\nabla V|^{p-2}\sum_{i=1}^N a_i\left(\frac{v_i'}{|x-y_i|} - v_i''\right)\sin^2\theta_i.
\end{align*}

It only remains to use the formula \eqref{fund.sol} for \(v'_i\) to compute that
\[\frac{v_i'}{|x-y_i|} - v_i'' = -c_{n,p}\frac{p+n-2}{p-1}|x-y_i|^\frac{2-n-p}{p-1}\]
and the sign of \(\Delta_p V\) can easily be read off the final identity
\begin{equation}
\Delta_p V(x) = -c_{n,p}\tfrac{(p-2)(p+n-2)}{p-1}|\nabla V|^{p-2}\sum_{i=1}^N a_i\frac{\sin^2\theta_i}{|x-y_i|^\frac{p+n-2}{p-1}}.
\label{sup.sign}
\end{equation}

\end{proof}

\begin{remark}
The three green lines in figure \ref{supfig} deserve some attention. The line \(p=2\) is obvious since the equation becomes linear. So is the line \(n=1\) as the ``angle'' between two numbers is 0 or \(\pi\). The little surprise, perhaps, is the case \(p+n = 2\). Then the terms in \( V\) will be on the form \(a_i|x-y_i|^2\) and it all reduces to the rather unexciting explanation that a linear combination of quadratics is again a quadratic.
\end{remark}

\section{Adding more terms}\label{add}
We will now examine what will happen to the sign of the \(p\)-Laplace operator when an extra term, \(K(x)\), is added to the linear combination \eqref{sup.lincomb}. We will from now on only consider \(p>2\). Restricted to this case, the factor \(C_{n,p} := c_{n,p}\frac{(p-2)(p+n-2)}{p-1}\) in \eqref{sup.sign} stays positive.

Let \( V\) be as in Lemma \ref{sup.thm} and let \(K\in C^2\). For efficient notation, write \(\xi = \xi(x) := \nabla V(x) + \nabla K(x)\). Then
\begin{align*}
\Delta_p( V + K) &= |\xi|^{p-2}\left((p-2)\frac{\xi\mathcal{H}( V + K)\xi^T}{|\xi|^2} + \Delta(V + K)\right)\\
                   &= |\xi|^{p-2}\left((p-2)\frac{\xi(\mathcal{H} V)\xi^T}{|\xi|^2} + \Delta V\right)\\
									 &\quad{}+ |\xi|^{p-2}\left((p-2)\frac{\xi(\mathcal{H}K)\xi^T}{|\xi|^2} + \Delta K\right).
\end{align*}
Now, the second to last term equals
\[-C_{n,p}|\xi|^{p-2}\sum_ia_i|x-y_i|^\frac{2-n-p}{p-1}\sin^2\alpha_i \leq 0\]
where \(\alpha_i\) is the angle between \(x-y_i\) and \(\nabla V(x) + \nabla K(x)\). Thus it suffices to ensure that the last term also is non-positive in order for the \(p\)-Laplace to hold its sign.
Lemma \ref{add.thm} presents a sufficient condition.

\begin{lemma}\label{add.thm}
Let \(p>2\) and define \(V\) as in \eqref{sup.lincomb}.
Then
\begin{equation}
\Delta_p( V(x) + K(x)) \leq 0
\label{eq:}
\end{equation}
for all concave functions \(K\in C^2(\mathbb{R}^n)\) wherever the left-hand side is defined.
%for all \(x\in \mathcal{D}( V)\cap\{\nabla V-\nabla g \neq 0\}\).
\end{lemma}

\begin{proof}
\(z^T(\mathcal{H}K)z \leq 0\) for all \(z\in\mathbb{R}^n\) since the Hessian matrix of a concave function \(K\) is negative semi-definite. Also \(K\) is superharmonic since the eigenvalues of \(\mathcal{H}K\) are all non-positive, i.e. \(\Delta K \leq 0\). Therefore,
\[\Delta_p( V(x) + K(x)) \leq |\xi|^{p-2}\left((p-2)\frac{\xi(\mathcal{H}K)\xi^T}{|\xi|^2} + \Delta K\right) \leq 0.\] 
\end{proof}

\begin{remark}
Though \(K\in C^2\) being concave is sufficient, it is not necessary. A counter example is provided by the quadratic form
\[K(x) = \frac{1}{2}x^TAx, \qquad\text{where } A=\diag(1-m,1,\dots,1),\; m=p+n-2.\]
Then \(K\) is not concave, but a calculation will confirm that \((p-2)\frac{\xi(\mathcal{H}K)\xi^T}{|\xi|^2} + \Delta K \leq 0\) and hence \(\Delta_p(V+K)\leq 0\). In fact, a stronger result than Lemma \ref{add.thm} is possible: Let \(f_i\) be \(C^2\) at \(x\) for \(i=1,\dots,N\) and let
\[\lambda_1^i \leq \lambda_2^i \leq \cdots \leq \lambda_n^i\]
be the eigenvalues of the Hessian matrix \(\mathcal{H}f_i(x)\). If
\[\lambda_1^i + \cdots + \lambda_{n-1}^i + (p-1)\lambda_n^i \leq 0\qquad \forall\, i,\]
then \(\Delta_p\left(\sum_i f_i\right) \leq 0\) at \(x\).
\end{remark}

\section{\(p\)-superharmonicity}\label{psup}

%\subsection{Setup}
We now prove that
\[W(x) :=\sum_{i=1}^\infty a_i w(x-y_i) + K(x),\qquad a_i \geq 0,\, y_i\in\mathbb{R}^n,\quad K \text{ concave}\]
is a \(p\)-superharmonic function in \(\mathbb{R}^n\).
The three cases \(2<p<n\), \(p=n\) and \(p>n\) are different and an additional assumption, \eqref{sup.assumption}, seems to be needed when \(p\geq n\). In the first case, only convergence at one point is assumed.  We start with the relevant definitions and a useful Dini-type lemma.

\begin{definition}
Let \(\Omega\) be a domain in \(\mathbb{R}^n\). A continuous function \(h\in W^{1,p}_{loc}(\Omega)\) is \textbf{\(p\)-harmonic} if
\begin{equation}
\int|\nabla h|^{p-2}\nabla h\nabla\phi^T\dd x = 0
\label{psup.pharmdef}
\end{equation}
for each \(\phi\in C_0^\infty(\Omega)\).
\end{definition}

%An equivalent definition is
%\begin{equation}
%\int|\nabla h|^p\dd x \leq \int|\nabla v|^p\dd x
%\label{psup.min}
%\end{equation}
%for all \(v\) such that \(v-h \in W^{1,p}_0(\Omega)\).

%\begin{definition}
%A function \( f \colon \Omega \rightarrow (-\infty,\infty]\) is \textbf{lower semi-continuous} (l.s.c.) in \(\Omega\) if the set
%\[\left\{x\in\Omega \;\middle |\; f(x) > \alpha\right\}\]
%is open for every \(\alpha\in\mathbb{R}\).
%\end{definition}

\begin{definition}\label{def.psup}
A function \( u \colon \Omega \rightarrow (-\infty,\infty]\) is \textbf{\(p\)-superharmonic} in \(\Omega\) if
\begin{enumerate}[i)]
	\item \( u \not\equiv \infty\).
	\item \( u\) is lower semi-continuous in \(\Omega\).
	\item If \(D\subset\subset\Omega\) and \(h\in C(\overline{D})\) is \(p\)-harmonic in \(D\) with \(h\big|_{\partial D}\leq  u\big|_{\partial D}\), then \(h\leq u\) in \(D\).
\end{enumerate}
\end{definition}
Furthermore, if \(u\in C^2(\Omega)\), it is a standard result that \(u\) is \(p\)-harmonic if and only if \(\Delta_p u = 0\) and \(u\) is \(p\)-superharmonic if and only if \(\Delta_p u\leq0\).

Also, a function \(u\) in \(C(\mathbb{R}^n)\cap W^{1,p}_{loc}(\mathbb{R}^n)\) is \(p\)-superharmonic if
\begin{equation}
\int_{\mathbb{R}^n}|\nabla u|^{p-2}\nabla u \nabla \phi^T\dd x \geq 0
\label{psup.inteq}
\end{equation}
for all \(0\leq \phi\in C^\infty_0(\mathbb{R}^n)\). See \cite{Lindqvist1986}.

\begin{lemma}\label{dini}
Let \((f_N)\) be an increasing sequence of lower semi-continuous (l.s.c.) functions defined on a compact set \(C\) converging point-wise to a function \(f\geq 0\). Then, given any \(\epsilon > 0\) there is an \(N_\epsilon\in\mathbb{N}\) such that
\[f_N(x) > - \epsilon\]
for all \(x\in C\) and all \(N\geq N_\epsilon\).
\end{lemma}
The standard proof is omitted.
%\begin{proof}
%Let \(\epsilon>0\). The set
%\[E_N := \left\{x\in C\;\middle |\; f_N(x) > - \epsilon\right\}\]
%is open for each \(N\) since \(f_N\) is l.s.c. Also \(E_N\subseteq E_{N+1}\) since \((f_N)\) is increasing. The collection \(\{E_N\}\) is an open cover of \(C\) since \(f_N(x)\to f(x) \geq 0\), so by compactness there is an integer \(M\) so that \(E_M = C\). Therefore, if \(x\in C\) and \(N\geq M\), then \(f_N(x) > -\epsilon\).
%\end{proof}

In the following, \(K\) is any concave function in \(\mathbb{R}^n\). We let \(K_\delta,\;\delta>0\) denote the smooth convolution \(\phi_\delta * K\) with some mollifier \(\phi_\delta\). One can show that
\(K_\delta\) is concave and
\[K_\delta \to K\]
locally uniformly on \(\mathbb{R}^n\) as \(\delta\to 0^+\).

\subsection{The case 2\(<\)p\(<\)n}

Let \(\delta>0\). If \(y_i\in\mathbb{R}^n\) and \(a_i>0\), the function
\[W^\delta_N(x) := \sum_{i=1}^N\frac{a_i}{|x-y_i|^\frac{n-p}{p-1}} + K_\delta(x)\]
is \(p\)-superharmonic except possibly at the poles \(y_i\) (Lemma \ref{add.thm}).
Defining \(W_N^\delta(y_i) := \infty\), \textbf{we claim that \(W^\delta_N\) is \(p\)-superharmonic in the whole \(\mathbb{R}^n\).}

We have to verify Def. \ref{def.psup}. Clearly, i) and ii) are valid. For the comparison principle in iii) we select \(D\subset\subset\mathbb{R}^n\) (i.e. \(D\) is bounded) and let \(h\in C(\overline{D})\) be \(p\)-harmonic in \(D\) with \(h\big|_{\partial D}\leq  W^\delta_N\big|_{\partial D}\).
If any, isolate the points \(y_i\) in \(\overline{D}\) with \(\epsilon\)-balls \(B_i := B(y_i,\epsilon)\) where \(\epsilon>0\) is so small so that \(W^\delta_N\big|_{B_i} \geq \max_{\overline{D}}h\).
This is possible because \(h\) is bounded and because \(\lim_{x\to y_i} W^\delta_N(x) = \infty\). Then \(W^\delta_N\) is \(C^2\) on \(D\setminus \cup B_i\) so, by Lemma ~\ref{add.thm}, \(\Delta_p W_N\leq0\) on this set. Also, \(h\big|_{\partial(D\setminus\cup B_i)} \leq  W^\delta_N\big|_{\partial(D\setminus\cup B_i)}\) by the construction of the \(\epsilon\)-balls, so \(h\leq W^\delta_N\) on this set since \( W^\delta_N\) is \(p\)-superharmonic there. Naturally, \(h\leq W^\delta_N\) on \(\cup B_i\), so the inequality will hold in the whole domain \(D\). This proves the claim.

Now \(N\to\infty\). Assume that the limit function
\[W^\delta(x) := \sum_{i=1}^\infty\frac{a_i}{|x-y_i|^\frac{n-p}{p-1}} + K_\delta(x)\]
is finite at least at one point in \(\mathbb{R}^n\). \textbf{We claim that \(W^\delta\) is \(p\)-superharmonic.}

By assumption \( W^\delta\not\equiv\infty\) and it is a standard result that the limit of an increasing sequence of l.s.c functions is l.s.c.

Part iii). Suppose that \(D\subset\subset\mathbb{R}^n\) and \(h\in C(\overline{D})\) is \(p\)-harmonic in \(D\) with \(h\big|_{\partial D}\leq  W^\delta\big|_{\partial D}\). Then \((W^\delta_N - h)\) is an increasing sequence of l.s.c. functions on the compact set \(\partial D\) with point-wise limit \((W^\delta-h)\big|_{\partial D}\geq 0\). If \(\epsilon>0\), then \((W^\delta_N - h)\big|_{\partial D} > - \epsilon\) for a sufficiently big \(N\) by Lemma \ref{dini}. That is
\[(h-\epsilon)\big|_{\partial D} < W^\delta_N\big|_{\partial D}\]
so \((h-\epsilon)\big|_{D} \leq W^\delta_N\big|_{D}\) since \(h-\epsilon\) is \(p\)-harmonic and \(W^\delta_N\) is \(p\)-superharmonic. Finally, since \(W^\delta_N\leq W^\delta\) we get
\[(h-\epsilon)\big|_D \leq W^\delta\big|_D\]
and as \(\epsilon\) was arbitrary, the required inequality \(h\leq W^\delta\) in \(D\) is obtained and the claim is proved.

Let \(\delta\to 0\) and set
\[W(x) := \sum_{i=1}^\infty\frac{a_i}{|x-y_i|^\frac{n-p}{p-1}} + K(x).\]
\textbf{We claim that \(W\) is \(p\)-superharmonic.}

Part i) and ii) are immediate. For part iii), assume \(D\subset\subset\mathbb{R}^n\) and \(h\in C(\overline{D})\) is \(p\)-harmonic in \(D\) with \(h\big|_{\partial D}\leq  W\big|_{\partial D}\). Let \(\epsilon>0\). Then there is a \(\delta>0\) such that
\[|K(x)-K_\delta(x)| < \epsilon\]
at every \(x\in\overline{D}\). We have
\[W^\delta = W + K_\delta - K > W - \epsilon \geq h - \epsilon\]
on \(\partial D\).
And again, since \(h-\epsilon\) is \(p\)-harmonic and \(W^\delta\) is \(p\)-superharmonic, we get \(W^\delta \geq h-\epsilon\) in \(D\). Thus
\[W\big|_D \geq W^\delta\big|_D-\epsilon \geq h\big|_D - 2\epsilon.\] 
This proves the claim, settles the case \(2<p<n\) and completes the proof of Theorem \ref{int.mainthm}.
\bigskip

We now turn to the situation \(p\geq n\) and introduce the assumption
\begin{equation}
A := \sum_{i=1}^\infty a_i<\infty.
\label{sup.assumption}
\end{equation}

\subsection{The case p=n}

Let \(\delta>0\). The partial sums
\[W^\delta_N(x) := -\sum_{i=1}^N a_i\ln|x-y_i| + K_\delta(x)\]
are \(p\)-superharmonic in \(\mathbb{R}^n\) by the same argument as in the case \(2<p<n\).

Let \(N\to\infty\). 
\textbf{We claim that}
\[W^\delta(x) := -\sum_{i=1}^\infty a_i\ln|x-y_i| + K_\delta(x)\]
\textbf{is \(p\)-superharmonic in \(\mathbb{R}^n\)} provided the sum converges absolutely\footnote{Conditional convergence is not sufficient. A counter example is \(a_i = ~1/i^2\), \(|y_i| = ~\exp((-1)^i i)\), yielding \(W^\delta(x) = -\infty\) for all \(y_i\neq x\neq0\). } at least at one point.

Assume for the moment that, given a radius \(R>0\), it is possible to find numbers \(C_i\) so that
\begin{equation}
\begin{gathered}
\ln|x-y_i| \leq C_i \text{ for all } x \in B_R := B(0,R), \text {and}\\
\text{the series } \sum_{i=1}^\infty a_i C_i =: S_R \text{ converges.}
\end{gathered}
\label{sup.c}
\end{equation}

Define the sequence \((f_N)\) in \(B_R\) by
\[f_N(x) := \sum_{i=1}^N\big(-a_i\ln|x-y_i| + a_iC_i\big) + K_\delta(x),\qquad f(x) := \lim_{N\to\infty}f_N(x).\]
Then \((f_N)\) is an increasing sequence of l.s.c functions implying that \(f\) is l.s.c. in \(B_R\) and that
\[W^\delta = f - S_R\]
is as well. Since \(R\) can be arbitrarily big, we conclude that \(W^\delta\) does not take the value \(-\infty\) and is l.s.c. in \(\mathbb{R}^n\).

For part iii) we show that \(f\) obeys the comparison principle. Assume \(D\subset\subset B_R\) and \(h\in C(\overline{D})\) is \(p\)-harmonic in \(D\) with \(h\big|_{\partial D}\leq  f\big|_{\partial D}\).
Then \((f_N - h)\) is an increasing sequence of l.s.c. functions on the compact set \(\partial D\) with point-wise limit
\[(f-h)\big|_{\partial D} \geq 0.\]
If \(\epsilon>0\), then \((f_N - h)\big|_{\partial D} > - \epsilon\) for a sufficiently big \(N\) by Lemma \ref{dini}. That is
\[(h-\epsilon)\big|_{\partial D} < f_N\big|_{\partial D}\]
so \((h-\epsilon)\big|_{D} \leq f_N\big|_{D}\) since \(h-\epsilon\) is \(p\)-harmonic and \(f_N\) is \(p\)-superharmonic. Finally, since \(f_N\leq f\) we get
\[(h-\epsilon)\big|_D \leq f\big|_D \]
and as \(\epsilon\) was arbitrary, the required inequality \(h\leq f\) in \(D\) is obtained. Hence \(W^\delta(x) = f(x) - S_R\) is a \(p\)-superharmonic function in any ball \(B_R\).

The claim is now proved if we can establish the existence of the numbers \(C_i\) satisfying \eqref{sup.c}.
By a change of variables we may assume that the convergence is at the origin. That is
\[L := \sum_{i=1}^\infty a_i|\ln|y_i|| < \infty.\]
We have
\begin{align*}
\ln|x-y_i| &\leq \ln(|x| + |y_i|)\\
           &\leq \ln(2\max\{|x|,|y_i|\})\\
					 &= \max\{\ln|x|,\ln|y_i|\} + \ln 2,
\end{align*}
so 
\[C_i := \max\{\ln R,\ln|y_i|\} + \ln 2\]
will do since (for \(R>1/2\)) the sequence of partial sums \(\sum_{i=1}^N a_i C_i\) is increasing and bounded by \(A\ln 2R + L\).

The final limit \(\delta\to 0\) causes no extra problems.
\[W(x) := -\sum_{i=1}^\infty a_i\ln|x-y_i| + K(x)\]
is \(p\)-superharmonic in \(\mathbb{R}^n\).

This settles the case \(p=n\).

\subsection{The case p\(>\)n}

Let \(\delta>0\). Consider again the partial sums
\[W^\delta_N(x) := -\sum_{i=1}^N a_i|x-y_i|^\frac{p-n}{p-1} + K_\delta(x).\]

As before \textbf{\(W^\delta_N\) is \(p\)-superharmonic in \(\mathbb{R}^n\)}, but now a different approach is required for the proof.
For ease of notation, write
\[u(x) := -\sum_{i=1}^N a_i|x-y_i|^\alpha + K(x),\qquad 0<\alpha := \frac{p-n}{p-1}<1,\]
where \(K\in C^\infty(\mathbb{R}^n)\) is concave.
We will show that \(u\) satisfies the integral inequality \eqref{psup.inteq}.

Clearly, \(u\) is continuous and \(\int_\Omega |u|^p\dd x < \infty\) on any bounded domain \(\Omega\). Also,
\[|\nabla (|x|^\alpha)|^p = \left|\alpha\frac{x^T}{|x|^{2-\alpha}}\right|^p \propto \frac{1}{|x|^{(1-\alpha)p}}\]
where one can show that
\[(1-\alpha)p < n.\]
Thus \(\int|\nabla u|^p\dd x < \infty\) locally so \(u \in C(\mathbb{R}^n)\cap W^{1,p}_{loc}(\mathbb{R}^n)\).

Let \(0\leq \phi\in C^\infty_0(\mathbb{R}^n)\) and write
\begin{align*}
\int_{\mathbb{R}^n}|\nabla u|^{p-2}\nabla u \nabla \phi^T\dd x &= \left(\int_{\mathbb{R}^n\setminus\cup_j B_j} + \int_{\cup_j B_j}\right)|\nabla u|^{p-2}\nabla u \nabla \phi^T\dd x\\
	&=: I_\epsilon + J_\epsilon
\end{align*}
where \(B_j := B(y_j,\epsilon)\) and where \(\epsilon>0\) is so small so that the balls are disjoint. Obviously, \(J_\epsilon\to 0\) as \(\epsilon\to 0\) but
\begin{align*}
I_\epsilon &= \int_{\partial(\mathbb{R}^n\setminus\cup_j B_j)}\phi |\nabla u|^{p-2}\nabla u\, \nu\dd\sigma - \int_{\mathbb{R}^n\setminus\cup_j B_j} \phi\Delta_p u\dd x\\
	&\geq \int_{\cup_j\partial B_j}\phi |\nabla u|^{p-2}\nabla u\, \nu\dd\sigma
\end{align*}
since \(\Delta_p u\leq 0\) on \(\mathbb{R}^n\setminus\cup_j B_j\) by Lemma \ref{add.thm}. Here, \(\nu\) is a sphere's \emph{inward} pointing normal so, for \(x\in\partial B_i\),
\begin{align*}
\nabla u(x)\nu &= \nabla u(x)\frac{y_i-x}{\epsilon}\\
  &= \left(-\alpha\sum_{j=1}^N a_j\frac{(x-y_j)^T}{|x-y_j|^{2-\alpha}} + \nabla K(x)\right)\frac{y_i-x}{\epsilon}\\
	&= \frac{\alpha a_i}{\epsilon^{1-\alpha}} + \alpha \sum_{j\neq i} a_j\frac{(x-y_j)^T}{|x-y_j|^{2-\alpha}}\frac{x-y_i}{\epsilon} + \nabla K(x)\frac{y_i-x}{\epsilon}\\
	&> \frac{\alpha a_i}{\epsilon^{1-\alpha}} - \frac{\alpha}{(d_i/2)^{1-\alpha}}\sum_{j\neq i} a_j - C_K,\qquad d_i := \min_{j\neq i}|y_j-y_i|\\
	&> 0
\end{align*}
for \(\epsilon\) sufficiently small.

That is,
\[\int_{\mathbb{R}^n}|\nabla u|^{p-2}\nabla u \nabla \phi^T\dd x \geq 0\]
for all non-negative test-functions. The partial sums are therefore \(p\)-superharmonic functions.

%We now use this integral inequality to show that \(u\) is \(p\)-superharmonic. Part i) and ii) are obvious, and for part iii) assume \(D\subset\subset\Omega\) and \(h\) is \(p\)-harmonic and continuous in \(\overline{D}\) with \(h\big|_{\partial D}\leq  u\big|_{\partial D}\).
%
%Define the function \(v := \max\{h-u,0\}\) on \(\overline{D}\). By a standard smoothing procedure one can show that the inequality \eqref{psup.inteq} still holds with the test function \(v\). That is
%\begin{align*}
%0 &\leq \int_{h>u}|\nabla u|^{p-2}\nabla u \nabla v^T\dd x\\
  %&= \int_{h>u}|\nabla u|^{p-2}\nabla u \nabla h^T\dd x - \int_{h>u}|\nabla u|^p\dd x.
%\end{align*}
%Applying Hölder's inequality yields
%\begin{align*}
%\int_{h>u}|\nabla u|^p\dd x &\leq \int_{h>u}|\nabla u|^{p-2}\nabla u \nabla h^T\dd x\\
                            %&\leq \int_{h>u}|\nabla u|^{p-1}|\nabla h|\dd x\\
														%&\leq \left(\int_{h>u}|\nabla u|^p\dd x\right)^\frac{p-1}{p}\left(\int_{h>u}|\nabla h|^p\dd x\right)^\frac{1}{p}
%\end{align*}
%which in turn implies that
%\[\int_{h>u}|\nabla u|^p\dd x \leq \int_{h>u}|\nabla h|^p\dd x.\]
%But \(h\) is \(p\)-harmonic and thus a minimizer on the open set \(\{h>u\}\) (cf. equation \eqref{psup.min}). The boundary values are \(u=h\) so \(u=h\) on this set since the minimizer is unique. This is a contradiction so \(\{h>u\} = \emptyset\) and the claim is proved.

Let \(N\to \infty\) and set
\[W^\delta(x) := -\sum_{i=1}^\infty a_i|x-y_i|^\alpha + K_\delta(x), \qquad \alpha := \frac{p-n}{p-1}\]
remembering the assumption \eqref{sup.assumption}.
This function is automatically \emph{upper} semi-continuous but as the definition of \(p\)-superharmonicity requires \emph{lower} semi-continuity, \emph{continuity} has to be shown.

\textbf{We claim that \(W^\delta\) is \(p\)-superharmonic in ~\(\mathbb{R}^n\)} provided the series converges at least at some point.

Again we may assume that the convergence is at the origin. That is \(\sum_{i=1}^\infty a_i|y_i|^\alpha  < \infty\). Since \(0<\alpha<1\), we get
\begin{align*}
|x-y_i|^\alpha &\leq (|x| + |y_i|)^\alpha\\
               &\leq |x|^\alpha + |y_i|^\alpha
\end{align*}
so since
\[\sum_{i=1}^\infty a_i|x-y_i|^\alpha \leq |x|^\alpha \sum_{i=1}^\infty a_i + \sum_{i=1}^\infty a_i|y_i|^\alpha < \infty\]
we see that \(W^\delta_N\to W^\delta\) locally uniformly in \(\mathbb{R}^n\). We infer that \(W^\delta\) is continuous in \(\mathbb{R}^n\).

For part iii), assume \(D\subset\subset\mathbb{R}^n\) and \(h\in C(\overline{D})\) is \(p\)-harmonic in \(D\) with \(h\big|_{\partial D}\leq  W^\delta\big|_{\partial D}\). Since \(W^\delta\leq W^\delta_N\) and \(W^\delta_N\) is \(p\)-superharmonic we get \(h\big|_D\leq  W^\delta_N\big|_D\) for all \(N\). So given any \(\epsilon>0\)
\[h\big|_D\leq  W^\delta\big|_D + \epsilon\]
by uniformity on the bounded set \(D\). This proves the claim.

Next, let \(\delta\to 0\). \textbf{Then}
\[W(x) := -\sum_{i=1}^\infty a_i|x-y_i|^\frac{p-n}{p-1} + K(x)\]
\textbf{is \(p\)-superharmonic in \(\mathbb{R}^n\)}
by the same argument as when \(2<p<n\). This settles the case \(p>n\).

\section{Epilogue: Evolutionary superposition.}\label{ep}

The superposition of fundamental solutions has been extended to \(p\)-Laplace equations in the Heisenberg group, see \cite{garofalo2010}. When it comes to further extensions, a natural question is whether such a superposition is valid for the evolutionary \(p\)-Laplace equation
\begin{align}
u_t &= \Delta_p u,\\
\intertext{or for the homogeneous equation}
\frac{\partial}{\partial t}(|u|^{p-2}u) &= \Delta_p u.
\end{align}
The following shows it does not.

In both cases \(p>2\) and \(u = u(x,t)\) where \(x\in\mathbb{R}^n\) and \(t>0\). The fundamental solutions to these equations
are given by
\[\mathcal{B}(x,t) := \frac{1}{t^{n\beta}}\left(C - \frac{p-2}{p}\beta^\frac{1}{p-1}\left(\frac{|x|}{t^\beta}\right)^\frac{p}{p-1}\right)_+^\frac{p-1}{p-2},\qquad \beta := \frac{1}{n(p-2) + p}\]
and
\[\mathcal{W}(x,t) := \frac{c}{t^{\frac{n}{p(p-1)}}}\exp\left(-\frac{p-1}{p}(1/p)^\frac{1}{p-1}\left(\frac{|x|}{t^{1/p}}\right)^\frac{p}{p-1}\right)\]
respectively, where the subscript \(+\) in the so-called Barenblatt solution \(\mathcal{B}(x,t)\) means \((\cdot)_+ = \max\{\cdot,0\}\). The \(C\) and \(c\) are positive constants chosen so that the solutions satisfy certain conservation properties.
For any fixed positive time the functions are \(C^2\) away from the origin and, in the case of \(\mathcal{B}\), away from the boudary of its support. We also notice that \(\mathcal{W}>0\) on \(\mathbb{R}^n\times(0,\infty)\) while \(\mathcal{B}\geq 0\) has compact support for any finite \(t\).

In some ways these functions are similar to the heat kernel. In particular, one can show that for any fixed \(0\neq y\in\mathbb{R}^n\) \emph{there is a time when the time derivatives} \(\mathcal{W}_t(y,t)\) and \(\mathcal{B}_t(y,t)\) change sign. In fact, a calculation will confirm that
\[\Delta_p(a\mathcal{B}) - (a\mathcal{B})_t = (a^{p-1} - a)\mathcal{B}_t,\qquad 0<a\neq 1\]
changes sign at \(y\) when
\[|y| = (Cpn)^\frac{p-1}{p}\beta^\frac{p-2}{p}t^\beta\]
showing that not even the simple superposition \(\mathcal{B} + \mathcal{B}\) holds.
This counter example arises due to \(\mathcal{B}\) not being multiplicative and will not work when applied to \(\mathcal{W}\).

Although the \(p\)-Laplacian
\[\Delta_p u = |\nabla u|^{p-2}\left((p-2)\frac{\nabla u(\mathcal{H} u)\nabla u^T}{|\nabla u|^2} + \Delta u\right),\qquad p>2,\]
is not well defined at \(x_0\) if \(\nabla u(x_0)=0\), it can be continuously extended to zero if \(u\) is \(C^2\) at the critical point. We will thus write \(\Delta_p u(x_0)=0\) in those cases.

Fix a non-zero \(y\in\mathbb{R}^n\) and define the linear combination \(V\) as
\begin{equation}
V(x,t) := \mathcal{W}(x+y,t) + \mathcal{W}(x-y,t).
\label{comb}
\end{equation}
Since \(\mathcal{W}(x,t) =: f(|x|,t)\) is radial in \(x\), the gradient can be written as
\[\nabla \mathcal{W}(x,t) = f_1(|x|,t)\frac{x^T}{|x|}\]
and
\[V(0,t) = \mathcal{W}(y,t) + \mathcal{W}(-y,t) = 2\mathcal{W}(y,t).\] 
Thus \(V\) is \(C^2\) at the origin and
\[\nabla V(0,t) = \bigg|_{x=0} f_1(|x+y|,t)\frac{(x+y)^T}{|x+y|} + f_1(|x-y|,t)\frac{(x-y)^T}{|x-y|} = 0\]
for all \(t>0\). So, at \(x=0\) we get
\begin{align*}
\frac{\partial}{\partial t}\left(|V|^{p-2}V\right) - \Delta_p V
	&= (p-1)V^{p-2}V_t - 0\\
	&= 2(p-1)(2\mathcal{W}(y,t))^{p-2}\mathcal{W}_t(y,t)
\end{align*}
which has the aforementioned change of sign at some time \(t\).
Thus the sum of the two fundamental solutions \(\mathcal{W}(x\pm y,t)\) cannot be a supersolution nor a subsolution.

\bibliographystyle{alpha}
\bibliography{C:/Users/Karl_K/Documents/PhD/references}

\begin{thebibliography}{LM08}

\bibitem[CZ03]{crandall2003}
Michael~G. Crandall and Jianying Zhang.
\newblock Another way to say harmonic.
\newblock {\em Trans. Amer. Math. Soc.}, 355(1):241--263 (electronic), 2003.

\bibitem[GT10]{garofalo2010}
N.~{Garofalo} and J.~{Tyson}.
\newblock {Riesz potentials and p-superharmonic functions in Lie groups of
  Heisenberg type}.
\newblock {\em ArXiv e-prints}, May 2010.

\bibitem[Lin86]{Lindqvist1986}
Peter Lindqvist.
\newblock On the definition and properties of p-superharmonic functions.
\newblock {\em Journal für die reine und angewandte Mathematik}, 365:67--79,
  1986.

\bibitem[LM08]{lindqvist2008}
Peter Lindqvist and Juan~J. Manfredi.
\newblock Note on a remarkable superposition for a nonlinear equation.
\newblock {\em Proc. Amer. Math. Soc.}, 136(1):133--140 (electronic), 2008.

\end{thebibliography}

\end{document}